\documentclass[12pt]{amsart}
\usepackage{breqn,cleveref}
\font\bbbld=msbm10 scaled\magstephalf

\newcommand{\bi}{\bar{i}}
\newcommand{\bj}{\bar{j}}
\newcommand{\bk}{\bar{k}}
\newcommand{\bl}{\bar{l}}
\newcommand{\bm}{\bar{m}}
\newcommand{\bn}{\bar{n}}

\newcommand{\bq}{\bar{q}}

\newcommand{\bt}{\bar{t}}

\newcommand{\balpha}{\bar{\alpha}}

\newcommand{\bzeta}{\bar{\zeta}}

\newcommand{\bpartial}{\bar{\partial}}

\newcommand{\fg}{\mathfrak{g}}
\newcommand{\fI}{\mathfrak{I}}

\newcommand{\fRe}{\mathfrak{Re}}

\newcommand{\bfN}{\hbox{\bbbld N}}

\newcommand{\bfR}{\hbox{\bbbld R}}

\newcommand{\cC}{\mathcal{C}}

\newcommand{\tr}{\mbox{tr}}

\newcommand{\ol}{\overline}

\newtheorem{theorem}{Theorem}[section]
\newtheorem{lemma}[theorem]{Lemma}
\newtheorem{proposition}[theorem]{Proposition}
\newtheorem{corollary}[theorem]{Corollary}

 \theoremstyle{definition}
\newtheorem{definition}[theorem]{Definition}

\theoremstyle{remark}
\newtheorem{remark}[theorem]{Remark}

\numberwithin{equation}{section}

\begin{document}

\title[fully nonlinear Parabolic equations]
{Convergence of a class of Fully non-linear Parabolic equations on Hermitian manifolds}
\author{Mathew George}
\address{Department of Mathematics, Ohio State University,
         Columbus, OH 43210, USA}
\email{george.924@buckeyemail.osu.edu}
%

\date{}

\begin{abstract}
We consider a class of fully non-linear parabolic equations on compact Hermitian manifolds involving symmetric functions of partial Laplacians. Under fairly general assumptions, we show the long time existence and convergence of solutions. We also derive a Harnack inequality for the linearized equation which is used in the proof of convergence.

{\em Mathematical Subject Classification (MSC2020):}
35K10, 35K55, 58J35.
\end{abstract}

\maketitle

\section{Introduction}
\label{gg-I}
\setcounter{equation}{0}
\medskip

In this paper we study the following parabolic equation on an $n$-dimensional compact Hermitian manifold $(M,\omega)$.

\begin{equation}
\label{I1}
\begin{aligned}
    &\frac{\partial \phi}{\partial t}= f(\Lambda(\sqrt{-1}\partial\bpartial \phi +X[\phi]))-\psi[\phi]\\
    &\phi(x,0)=\phi_0\in C^{\infty}(M)
    \end{aligned}
\end{equation}

\noindent where $f(\Lambda)$ is a symmetric function of $\Lambda_i$ which denotes a partial sum of eigenvalues of $\sqrt{-1}\partial\bpartial \phi +X[\phi]$. More precisely, let  $K \leq n$ be a fixed positive integer. Set 
\[ \fI_K = \{(i_1, \ldots, i_K): 1 \leq i_1 < \cdots < i_K \leq n, \; i_j \in \bfN\} \]

\noindent Denote the elements of $\fI_K$ by $\{I_1,\hdots,I_N\}$ after fixing an order. Then

\[ \Lambda (\lambda) = (\Lambda_{1} (\lambda), \ldots, \Lambda_{N} (\lambda)):=(\Lambda_{I_1} (\lambda), \ldots, \Lambda_{I_N} (\lambda)) \]

\noindent where

\[ \Lambda_I (\lambda) 
    = \sum_{i \in I} \lambda_i = \lambda_{i_1} + \cdots + \lambda_{i_K}, 
    \;\; \lambda =  (\lambda_{1},  \ldots, \lambda_{n}) \in \bfR^n. \]

\noindent We also write
\[ \Lambda (\sqrt{-1} \partial \bpartial \phi + X [\phi]) := \Lambda (\lambda(\sqrt{-1} \partial \bpartial \phi + X [\phi])) \]
\noindent where $ \lambda(\sqrt{-1} \partial \bpartial \phi + X [\phi]) = (\lambda_{1},  \ldots, \lambda_{n})$ denotes the eigenvalues of $\sqrt{-1} \partial \bpartial \phi + X [\phi]$ with respect to 
$\omega$.

This form of $\Lambda$ was introduced in \cite{GGQ21} as a generalization of $\lambda(\Delta u\omega -\sqrt{-1}\partial\bpartial u+X[u])$ which is obtained in the case when $K=n-1$. The $(1,1)$-form $X[\phi]=X(z, \phi, \partial \phi, \bpartial \phi)$ and the function $\psi[\phi]=\psi(z, \phi, \partial \phi, \bpartial \phi)$ depend on $\phi$ and its first order derivatives as well. Throughout this article $f$, $X$ and $\psi$ are assumed to be smooth.

Equation \eqref{I1} is the parabolic counterpart of an elliptic equation studied by the author with Guan and Qiu in \cite{GGQ21} on Hermitian manifolds. Such equations are of interest, for example in the proof of Gauduchon conjecture by  Sz\'ekelyhidi-Tosatti-Weinkove \cite{STW17} and in the work of Guan-Qiu-Yuan \cite{GQY19} involving the study of conformal deformations of mixed Chern-Ricci forms. 

\

To state the main theorem we make the following set of assumptions on $f$, $X$ and $\psi$.

\

\noindent \textbf{Assumptions on $f$:} $f$ is a symmetric function of $N$ variables defined in a symmetric open convex cone $\Gamma\subset \mathbb{R}^N$ with vertex at the origin with

\begin{equation}\label{P1}
    \Gamma_N=\{\Lambda\in \mathbb{R}^N:\Lambda_i>0\}\subset\Gamma,
\end{equation}

\noindent and satisfies the conditions

\begin{equation}\label{P2}
    f_i\equiv \frac{\partial f}{\partial \Lambda_i}\geq 0 \text{ in $\Gamma$, $1\leq i\leq N$,}
\end{equation}

\begin{equation}\label{P3}
    \text{$f$ is a concave function in $\Gamma$},
\end{equation}

\begin{equation}\label{P4}
    \sup\limits_{\partial\Gamma}f<\inf\limits_M\psi
\end{equation}

\noindent and,

\begin{equation}\label{P5}
    \lim\limits_{t\to \infty}f(t\Lambda)=\sup\limits_{\Gamma}f, \;\; \forall \;\Lambda\in \Gamma
\end{equation}

We will say that $\phi$ is an \textit{admissible} function if $\Lambda (\sqrt{-1} \partial \bpartial \phi + X [\phi])\in \Gamma$, for all $t$.

\begin{remark}
With strict inequality in \eqref{P2}, conditions \eqref{P1} to \eqref{P5} are the structure conditions of Caffarelli-Nirenberg-Spruck \cite{CNS85}.
\end{remark}

For deriving first and second order estimates, we will make the following additional assumptions on $f$.
\begin{equation}\label{S0.1}
    \sum f_i\Lambda_i\geq -C_0\sum f_i \text{ in $\Gamma$}
\end{equation}

\noindent for some constant $C_0>0$.

\begin{equation}\label{S0.2}
    \text{rank of $\mathcal{C}_{\sigma}^+\geq \frac{N(n-K)}{n}+1, \; \forall \; \inf\limits_{\Gamma}f\leq \sigma \leq \sup\limits_{\partial\Gamma}f$ }
\end{equation}

\noindent and,

\begin{equation}\label{S0.3}
    \lim\limits_{t\to+\infty}f(t\mathbf{1})-\sup\limits_M\psi[\phi]\geq c_0>0.
\end{equation}

It is worth emphasizing that assumption \eqref{S0.2} is the critical ingredient used in the derivative estimations. Condition \eqref{S0.1} means that at any point $\Lambda_0\in \Gamma$ , the distance from the origin to the tangent plane at $\Lambda_0$ of the level hypersuface $\partial \Gamma^{f(\Lambda_0)}$ has a uniform
bound $C_0$. It is satisfied in most applications and is weaker than assumption \eqref{P5} which implies 

$$\sum f_i\Lambda_i \geq 0 \text{ in } \Gamma$$

\

\noindent \textbf{Assumptions on $\psi$ and $X$ :} Some growth assumptions must be made on $\psi$ and $X$. Here we impose the following conditions. 

\begin{equation}\label{U0.1}
    G^{i\bj}X_{i\bar j,\phi}-\psi_{\phi}\leq 0
\end{equation}
\noindent where $G^{i\bj}$ are the coefficients of second-order terms in the linearized equation (see section \ref{E}). This is only used for estimating $\sup{|\phi_t|}$ by the maximum principle. In section \ref{C}, we will assume that $X$ and $\psi$ are independent of $\phi$ for proving the convergence in Theorem \ref{theorem-I1}.
 
 \

 For deriving gradient estimates, the following conditions are assumed.
 
 \begin{equation}\label{G0.1}
    |D_{\zeta}X(z,\phi,\zeta,\overline{\zeta})|\leq \varrho_0|\zeta|,\; D_\phi X\leq (\varrho_0|\zeta|^2+\varrho_1)\omega
\end{equation}

\noindent where $\varrho_1=\varrho_1(z,\phi)$ and $\varrho_0=\varrho_0(z,\phi,|\zeta|)\to 0^+$ as $|\zeta|\to\infty$; we may assume $t\varrho_0(z,\phi,t)$ to be increasing in $t>0$. It follows that $|X|\leq C\varrho_0|\zeta|^2+\varrho_1(z,\phi)$, for some function $\varrho_1$; we shall only need 

\begin{equation}\label{G0.2}
    X\leq (\varrho_0|\zeta|^2+\varrho_1)\omega
\end{equation}

On $\psi$ we impose similar constraints, but also depending on the growth of $f$.

\begin{equation}\label{G0.3}
    |D_{\zeta}\psi(z,\phi,\zeta,\overline{\zeta})|\leq \varrho_0f(|\zeta|^2\mathbf{1})/|\zeta|,\; -D_{\phi}\psi\leq \varrho_0f(|\zeta|^2\mathbf{1})+\varrho_1(z,\phi)
\end{equation}

\noindent which implies

\begin{equation}\label{G0.4}
    \psi\leq \varrho_0f(|\zeta|^2\mathbf{1})+\varrho_1(z,\phi)
\end{equation}

We will also assume that 

\begin{equation}\label{G0.5}
    |\nabla_z X |\leq |\zeta| \left(\varrho_0f(|\zeta|^2\mathbf{1})+\varrho_1(z,\phi)\right) 
\end{equation}

and,

\begin{equation}\label{G0.6}
    |\nabla_z \psi|\leq  |\zeta|\left(\varrho_0|\zeta|^2+\varrho_1(z,\phi)\right)
\end{equation}

We will assume for convenience that $\int_M\omega^n=1$. Then the main result is stated as follows. 

\begin{theorem}\label{theorem-I1}
Let $f$, $X$ and $\psi$ satisfy (\ref{P1})-(\ref{P5}), (\ref{S0.2})-(\ref{G0.1}), (\ref{G0.3}), (\ref{G0.5}) and (\ref{G0.6}). Then equation (\ref{I1}) has an admissible solution $\phi$ for all time $t\in [0,\infty)$. In addition, the normalized function of $\phi$ defined by

\begin{equation}\label{I2}
    \bar{\phi}:=\phi-\int_M\phi\omega^n
\end{equation}

\noindent has the following uniform estimate

$$|\bar{\phi}|_{C_{x,t}^{\infty,1}}\leq C$$ 
for a constant $C$ that depends only on $(M,\omega)$ and $\phi_0$. If $X$ and $\psi$ are independent of $\phi$, then $\bar{\phi}$ converges in $C^{\infty}$ to a smooth function $\bar{\phi}_{\infty}$ as $t\to \infty$, where $\bar{\phi}_{\infty}$ is a solution of the elliptic equation
    \begin{equation}\label{I4}
        f(\Lambda(\sqrt{-1}\partial\bpartial u +X[u]))=\psi[u]+a
    \end{equation}
            \noindent for some constant $a$.

\end{theorem}

\
\begin{remark}
If $\phi$ is a solution of (\ref{I1}), then $\bar{\phi}$ solves
\begin{equation}\label{I3}
\begin{aligned}
    &\frac{\partial \bar{\phi}}{\partial t}= f(\Lambda(\sqrt{-1}\partial\bpartial {\bar{\phi}} +X[{\phi}]))-\psi[{\phi}]-\int_M\frac{\partial\phi}{\partial t}\omega^n\\
    &\bar{\phi}(x,0)=\phi_0-\int_M\phi_0\omega^n
\end{aligned}
\end{equation}
\end{remark}

 \begin{remark}
 It would be interesting to investigate whether the convergence in Theorem (\ref{theorem-I1}) holds with $X$ and $\psi$ depending on $\phi$.
 \end{remark}

The organization of the paper is as follows. To prove the long-time existence of solutions, we will derive apriori estimates for the $C^{2,\alpha}(M)$ norm of $\bar{\phi}$. The second-order and gradient estimates for $\bar{\phi}$ will be derived in sections \ref{S} and \ref{G} respectively. This will be done by applying maximum principle to test functions similar to the elliptic case. These estimates will depend on $\sup{|\phi_t|}$ which is easily bounded by applying maximum principle to the linearized equation. This is done in section \ref{E}.

Gradient estimates in section \ref{G} will in turn provide a uniform estimate for $\bar{\phi}$. All of this combined with Evans-Krylov theorem and a standard bootstrapping argument will imply the apriori estimate for $\bar{\phi}$ in Theorem \ref{theorem-I1}. Now the solution can be extended to $T=\infty$, which will be shown in section \ref{L}.

For proving the convergence, we will derive a Harnack inequality for positive solutions of equations of the form,

\begin{equation}
    \frac{\partial u}{\partial t}=  G^{i\bar j}\partial_i\partial_{\bar j}u+\chi_ku_k+\chi_{\bar{k}}u_{\bar k}+\chi_0 u
\end{equation}

This is similar to the results of Li-Yau \cite{LY86} and Gill \cite{Gill11}, but on Hermitian manifolds and with lower order terms. Li and Yau considered parabolic equations associated to the Schr\"odinger operator given by

\begin{equation}\label{li-yau}
u_t=\Delta u -q(x,t)u  
\end{equation}

\noindent on a Riemannian manifold $M$ with $q \in C_{x,t}^{2,1}(M,[0,T))$. On the other hand, Gill derived Harnack inequality for

\begin{equation}\label{gill}
   u_t=g^{i\bj}u_{i\bj} 
\end{equation}

\noindent on compact Hermitian manifolds, for proving the convergence of Chern-Ricci flow in the case when $c^{BM}_1(M)=0$. This is also related to the work of H.D. Cao \cite{Cao1985} who considered the K\"ahler version of the same equation. Our result uses similar techniques as the above equations, but now the estimation is more complicated because of the additional terms involved.

The Harnack inequality can further be used to derive an exponential decay for the oscillation $\omega(t)$ of $\phi_t$. From there, the convergence follows by standard arguments as detailed in section 8.

\section{Preliminaries}

\label{P}
\setcounter{equation}{0}
\medskip

Now we shall introduce the basic notations and state some key lemmas that will be used in the subsequent sections. Let $C^{k,p}_{x,t}(M\times I)$ be the set of functions defined on $M\times I$ whose derivatives of orders up to $(k,p)$ in $(x,t)$ variables exist and are continuous.

\medskip

We recall some notions from \cite{GGQ21}, \cite{G14} and \cite{GN21}. For a fixed real number $\sigma \in (\sup_{\partial \Gamma} f, \sup_{\Gamma} f)$
define
\[ \Gamma^{\sigma} = \{\lambda \in \Gamma: f (\lambda) > \sigma\}. \]

\begin{lemma}[\cite{GGQ21}]\label{lemma1}
 Under conditions \eqref{P2} and \eqref{P3}, the level hypersurface of $f$
 \[ \partial \Gamma^{\sigma} = \{\lambda \in \Gamma: f (\lambda) = \sigma\}, \]
which is the boundary of $\Gamma^{\sigma}$, is smooth and convex. 
\end{lemma}

This is clearly true with strict inequality in \eqref{P2}, but still remains valid under the slightly weaker hypothesis.

\

Define for $\lambda \in \partial \Gamma^{\sigma}$,

\[ \nu_{\lambda} = \frac{Df (\lambda)}{|Df (\lambda)|} \]

\

 $\nu_{\lambda}$ is the unit normal vector to $\partial \Gamma^{\sigma}$ at $\lambda$. The key ingredient used in finding apriori estimates later on is obtained by studying the \textit{tangent cone at infinity} to the level sets of $f$.

\begin{definition}[\cite{G14}] 
For $\mu \in \bfR^n$
let 
 \[ S^{\sigma}_{\mu} = \{\lambda \in \partial
\Gamma^{\sigma}: \nu_{\lambda} \cdot (\mu - \lambda) \leq 0\}. \]
The  {\em tangent cone at infinity} 
to $\Gamma^{\sigma}$ is defined as
\[ \begin{aligned}
\cC^+_{\sigma}
 \,& = \{\mu \in \bfR^n:
              S^{\sigma}_{\mu} \; \mbox{is compact}\}.
    \end{aligned} \]
\end{definition}

Clearly $\cC^+_{\sigma}$ is a symmetric convex cone. As in \cite{G14} one can show that $\cC^+_{\sigma}$ is open.

\begin{definition}[\cite{GN21}]
The {\em rank of ${\mathcal{C}}_{\sigma}^+$} is defined to be
\[ \min \{r (\nu): \mbox{$\nu$ is the unit normal vector of a supporting plane
to ${\mathcal{C}}_{\sigma}^+$}\} \]
where $r (\nu)$ denotes the number of non-zero
components of $\nu$.
\end{definition}


Under the assumptions \eqref{P2}, \eqref{P3} and that $f$ satisfies,

\begin{equation}\label{P6}
    \sum f_i\Lambda_i\geq -C_0\sum f_i \text{  in $\Gamma$},
\end{equation}

\noindent we have the following results from \cite{GGQ21}.
\begin{lemma}
\label{lemma 2}
Let $P = \{\mu \in \bfR^N: \nu \cdot \mu = c\}$ be a hyperplane,  where $\nu$ is a unit vector. Suppose that 
there exists a sequence 
$\{\Lambda_k\}$ in $\partial \Gamma^{\sigma}$ with 
\begin{equation}
\label{P7}
\lim_{k \rightarrow + \infty} \nu_{\Lambda_k} = \nu, \;\;
\lim_{k \rightarrow + \infty} \nu_{\Lambda_k} \cdot \Lambda_k = c,  \;\;
\lim_{k \rightarrow + \infty} |\Lambda_k| = + \infty.  
\end{equation}
Then $P$ is a supporting hyperplane to $\mathcal{C}_{\sigma}^+$ at a non-vertex point.
\end{lemma}

\begin{lemma}
\label{lemma 3} 
Suppose that the rank of $\mathcal{C}_{\sigma}^+$ is $r$. There exists $c_0 > 0$ 
such that at any point $\Lambda \in \partial \Gamma^{\sigma}$ where without loss of generality we assume $f_1 \leq \cdots \leq f_n$, 
\[ \sum_{i \leq N-r+1} f_i \geq c_0 \sum f_i. \]
\end{lemma}

Proofs of these statements can be found in \cite{GGQ21}.

\begin{remark} An inequality of Lin-Trudinger~\cite{LT94} shows that
the rank of $\mathcal{C}_{\sigma}^+$ is $N-k+1$ for $f = \sigma_k^{\frac{1}{k}}$
and $\sigma > 0$, where
\[ \sigma_k (\Lambda) = \sum_{1 \leq i_1 < \cdots < i_k \leq N}
     \Lambda_{i_1} \cdots \Lambda_{i_k} \] 
is the  $k$-th elementary symmetric function defined on the Garding cone
\[ \Gamma_k = \{\Lambda \in \bfR^N: \sigma_j (\Lambda) > 0, \;
\mbox{for $1 \leq j \leq k$}\}. \]

\end{remark}

\medskip

\section{Estimate for $\phi_t$}
\label{E}

To estimate $\sup{|\phi_t|}$, we will apply maximum principle to the linearization of \eqref{I1}.

\begin{proposition}\label{prop1}
Let $\phi\in C^{2,1}_{x,t}(M\times [0,T))$ be a solution of \eqref{I1}. Then under the assumption \eqref{U0.1},

$$\sup{|\phi_t|}\leq C$$
for some constant $C$ only depending on the initial data.
\end{proposition}

\begin{proof}

 Denote
$$\fg[\phi]:=\sqrt{-1}\partial\bpartial \phi+ X[\phi],$$
and define $G$ by 
$$G(\fg[\phi])=f(\Lambda(\fg[\phi]))$$ Consequently, we can write \eqref{I1} as 
\label{U}
\setcounter{equation}{0}
\medskip
\begin{equation}
\label{U1}
    \frac{\partial \phi}{\partial t}= G(\fg)-\psi[\phi]
\end{equation}

Differentiate the equation wrt. $t$ and denote $\phi_t\equiv u$.
\begin{equation}
    \label{U2}
\begin{aligned}
    \frac{\partial u}{\partial t}&=G^{i\bar j}\partial_t \mathfrak g_{i\bar j}-\partial_t \psi[\phi]\\
    &= G^{i\bar j}\partial_i\partial_{\bar j}u+G^{i\bj}\partial_tX_{i\bar j}[\phi]-\partial_t\psi[\phi]
\end{aligned}
\end{equation}

\noindent where $G^{i\bj}=\dfrac{\partial G}{\partial \fg_{i\bj}}(\fg[\phi])$. Expanding the last two terms by chain rule gives,

\begin{equation}
\label{U3}
\begin{aligned}
    G^{i\bj}\partial_tX_{i\bar j}[\phi]-\partial_t\psi[\phi]&=G^{i\bj}X_{i\bar j,\phi}u+G^{i\bj}X_{i\bar j,\zeta_k}u_k+G^{i\bj}X_{i\bar j,\zeta_{\bar k}}u_{\bar k}-(\psi_{\phi}u+\psi_{\zeta_k}u_k+\psi_{\zeta_{\bar k}}u_{\bar k})\\
    &:= \chi_k[\phi]u_k+\chi_{\bar{k}}[\phi]u_{\bar k}+\chi_0[\phi]u
\end{aligned}
\end{equation}
 So we get the linearization of equation \eqref{I1} given by

\begin{equation}
\label{U4}
      \frac{\partial u}{\partial t}=  G^{i\bar j}\partial_i\partial_{\bar j}u+\chi_k[\phi]u_k+\chi_{\bar{k}}[\phi]u_{\bar k}+\chi_0[\phi]u
\end{equation}

Assuming $G^{i\bj}X_{i\bar j,\phi}-\psi_{\phi}\leq 0$, that is $\chi_0\leq 0$, allows us to apply the parabolic maximum principle (see \cite{Lieberman} Theorem 7.1) to show that $\sup{|u|}$ is uniformly bounded. Note that the infimum of $u$ is bounded by applying maximum principle to $-u$. 

\end{proof}

\section{Second Order Estimates}
\label{S}
\setcounter{equation}{0}
\medskip

We introduce some geometric preliminaries first. Throughout this article $\nabla$ will denote the Chern connection with respect to the metric $\omega$. This is the unique connection on $TM$ defined by 

$$d \hspace{1pt}\langle s_1,s_2\rangle=\langle\nabla s_1, s_2\rangle+\langle s_1,\nabla s_2\rangle \text{ and } \nabla^{0,1}=\bpartial$$

\noindent for any two smooth sections $s_1$ and $s_2$ of $TM$ and $\nabla^{0,1}$ denotes the projection of $\nabla$ onto $T^{0,1}M$. In local coordinates this can be written as

$$\nabla_i\partial_j:=\Gamma^{k}_{ij}\partial_k, \text{ where } \Gamma_{ij}^k=g^{k\bl}\partial_ig_{j\bl}$$

The torsion and the curvature tensors are

$$T_{ij}^l=\Gamma^l_{ij}-\Gamma^l_{ji}$$

\noindent and,
$$R_{i\bj k\bl }=-\partial_{\bj}\partial_i g_{k\bl}+g^{p\bq}\partial_ig_{k\bq}\partial_{\bj}g_{p\bl}$$

\noindent respectively.

\

For a function $f$, denote $f_{i\bj}=\nabla_{\bj}\nabla_i f$, $f_{ij}=\nabla_j\nabla_i f$ etc. Then we have the following equations for commuting covariant derivatives,

\begin{equation}\label{commute}
    \begin{aligned}
        &f_{i\bj k}-f_{ik \bj}= -g^{l\bm} R_{k\bj i \bm}f_l,\\
        &f_{i\bj\bk}-f_{i\bk \bj}=\overline{T_{jk}^l}f_{i\bl},\\
        &f_{i\bj k}-f_{ki\bj}=-g^{l\bm}R_{i\bj k\bm}f_l+T^l_{ik}f_{l\bj},\\
        &f_{i\bj k\bl}-f_{k\bl i\bj}=g^{p\bq}(R_{k\bl i \bq}f_{p\bj}-R_{i\bj k \bq}f_{p\bl})+T^p_{ik}f_{p\bj\bl}+\overline{T^q_{jl}}f_{i\bq k}-T^p_{ik}\overline{T^q_{jl}}f_{p\bq}
    \end{aligned}
\end{equation}

\

As in section \ref{U}, we write \eqref{I1} as 

\begin{equation}
    \begin{aligned}
        \frac{\partial \phi}{\partial t}= G(\fg)-\psi[\phi]
    \end{aligned}
\end{equation}
Differentiating this equation with respect to $z_i$ first and then wrt. $z_{\bj}$ gives,

\begin{equation}\label{diff-pde}
    \begin{aligned}
        \partial_i \phi_t&=G^{p\bq}\nabla_i\fg_{p\bq}-\nabla_{i}\psi[\phi]\\
        \partial_i\partial_{\bj}\phi_t&=G^{p\bq,s\bt}\nabla_{\bj}\fg_{s\bt}\nabla_i\fg_{p\bq}+G^{p\bq}\nabla_{\bj}\nabla_i\fg_{p\bq}-\nabla_{\bj}\nabla_i\psi[\phi]\\
        &\leq G^{p\bq}\nabla_{\bj}\nabla_i\fg_{p\bq}-\nabla_{\bj}\nabla_{i}\psi[\phi]
    \end{aligned}
\end{equation}
\noindent where the last inequality follows from concavity of $f$, and
$$G^{p\bq,s\bt}:=\dfrac{\partial^2 G}{\partial \fg_{p\bq}\partial\fg_{s\bt}}(\fg[\phi])$$

From now on we assume that $f$ satisfies \eqref{S0.1}, \eqref{S0.2} and \eqref{S0.3}.

\noindent where $\mathbf{1}=(1,\hdots,1)\in \Gamma$ and $c_0$ may depend on $|\phi|_{C^1(M)}$. 

We estimate $|\partial\bpartial \phi|$ by following \cite{GGQ21} which uses ideas of Tossati-Weinkove \cite{TW17} and consider the test function which is given in local coordinates by
\begin{equation}\label{S1}
\begin{aligned}
A:=\sup\limits_{(z,t)\in M\times [0,T)}\max\limits_{\xi\in T_z^{1,0}M}e^{(1+\gamma)\eta}\fg_{p\bq}\xi_p\overline{\xi}_q(g^{k\bl}\fg_{i\bl}\fg_{k\bj}\xi_i\overline{\xi}_j)^{\frac{\gamma}{2}}/|\xi|^{2+\gamma}
\end{aligned}
\end{equation}

\noindent where $\eta$ is a function depending on $|\nabla\phi|$, and $\gamma>0$ is a small constant to be chosen later.

\begin{theorem}
Let $\phi\in C^{4,1}_{x,t}(M\times[0,T))$ be a solution to \eqref{I1}. Then

$$\sup_{M\times [0,T)}|\partial \bpartial \phi|_g\leq C$$ 

\noindent where $C$ depends on $(M,\omega)$, $\sup|\nabla \phi|$ and $\sup{|\phi_t|}$. 
\end{theorem}

\begin{proof}
Assume that $A$ is acheived at a point $(z_0,t_0)\in M\times [0,T)$ for some $\xi\in T_{z_0}^{1,0}M$. We choose local coordinates around $z_0$ such that $g_{i\bj}=\delta_{ij}$ and $T_{ij}^k=2\Gamma^k_{ij}$ using the lemma of Streets and Tian \cite{ST09}, and that $\fg_{i\bj}$ is diagonal at $z_0$ with $\fg_{1\bar{1}}\geq \fg_{2\bar{2}}\geq\hdots \geq \fg_{n\bn}$.

The maximum $A$ is achieved for $\xi=\partial_1$ at $(z_0,t_0)$ when $\gamma>0$ is sufficiently small (see \cite{TW17} and \cite{GN21}). We can assume $\fg_{1\bar{1}}>1$; otherwise we are done. 

Let $W = g_{1\bar{1}}^{-1} g^{k\bl} \fg_{1\bl} \fg_{k\bar{1}}$.
We see that the function
$Q=(1+\gamma)  \eta+\log{ g_{1\bar{1}}^{-1} \fg_{1\bar{1}}}+ \dfrac{\gamma}{2}\log{W}$ which is  locally well defined
attains a maximum $(1+\gamma){\eta}+(1+\gamma)\log{ \fg_{1\bar{1}}}$ at $(z_0,t_0)$ where
$W = \fg_{1\bar{1}}^2$ and the following equations are obtained.
\begin{equation}
\label{S2}
 \begin{aligned}
\frac{\partial_i (g_{1\bar{1}}^{-1} \fg_{1\bar{1}})}{\fg_{1\bar{1}}}
   +  \frac{\gamma \partial_i W}{2 W} + (1+ \gamma)  \partial_i \eta = \,& 0, \\
\frac{\bpartial_i (g_{1\bar{1}}^{-1} \fg_{1\bar{1}})}{\fg_{1\bar{1}}}
   +\frac{ \gamma \bpartial_i W}{2 W} + (1+ \gamma) \bpartial_i \eta = \,& 0
\end{aligned}
\end{equation}
for each $1 \leq i \leq n$, and
\begin{equation}
\label{S3}
\begin{aligned}
0 \geq \,&
   \frac{1}{\fg_{1\bar{1}}} G^{i\bi} \bpartial_i \partial_i (g_{1\bar{1}}^{-1} \fg_{1\bar{1}})
   - \frac{1}{\fg_{1\bar{1}}^2} G^{i\bi}
     \partial_i (g_{1\bar{1}}^{-1} \fg_{1\bar{1}}) \bpartial_i (g_{1\bar{1}}^{-1} \fg_{1\bar{1}}) \\
  & + \frac{\gamma}{2 W} G^{i\bi} \bpartial_i \partial_i W
     - \frac{\gamma}{2 W^2} G^{i\bi} \partial_i W \bpartial_i W
   + (1 + \gamma) G^{i\bi} \bpartial_i \partial_i \eta.
\end{aligned}
\end{equation}

The following identities can be derived by direct calculation.

\begin{equation}
\label{S4}
\begin{aligned}
\partial_i (g_{1\bar{1}}^{-1} \fg_{1\bar{1}}) = \nabla_i \fg_{1\bar{1}}, \;\;
 \partial_i W 
 = 2 \fg_{1\bar{1}} \nabla_i \fg_{1\bar{1}},
     \end{aligned}
      \end{equation}
\begin{equation}
\label{S5}
 \begin{aligned}
\bpartial_j  \partial_i (g_{1\bar{1}}^{-1} \fg_{1\bar{1}})
       = \,& \nabla_{\bj} \nabla_i \fg_{1\bar{1}}
    + (\ol{\Gamma_{j1}^m} \nabla_i \fg_{1\bm}
        - \ol{\Gamma_{j1}^1} \nabla_i \fg_{1\bar{1}}) \\
      & + (\Gamma_{i1}^m \nabla_{\bj} \fg_{m\bar{1}}
            - \Gamma_{i1}^1 \nabla_{\bj} \fg_{1\bar{1}})
      + (\Gamma_{i1}^1  \ol{\Gamma_{j1}^1} - \Gamma_{i1}^m \ol{\Gamma_{j1}^m}) \fg_{1\bar{1}}.
     \end{aligned}  
 \end{equation}     
      and
\begin{equation}
\label{S6}
 \begin{aligned}
\bpartial_j \partial_i W
   = \,& 2 \fg_{1\bar{1}} \nabla_{\bj} \nabla_i \fg_{1\bar{1}}
   + 2 \nabla_i \fg_{1\bar{1}} \nabla_{\bj} \fg_{1\bar{1}}
   + \sum_{l>1} \nabla_i \fg_{l\bar{1}} \nabla_{\bj} \fg_{1\bl}    \\
     &  + \sum_{l > 1} (\nabla_i \fg_{1\bl} + {\Gamma_{i1}^l} \fg_{l\bl})
       (\nabla_{\bj} \fg_{l\bar{1}} + \ol{\Gamma_{j1}^l} \fg_{l\bl}) \\
   & + \fg_{1\bar{1}} \sum_{l > 1} (\ol{\Gamma_{j1}^l} \nabla_i \fg_{1\bl}
       + \Gamma_{i1}^l  \nabla_{\bj} \fg_{l\bar{1}}) \\
  &    -  \fg_{1\bar{1}} \sum_{l > 1} \Gamma_{i1}^m \ol{\Gamma_{j1}^m} (\fg_{1\bar{1}} +  \fg_{l\bl)}.
              \end{aligned}
 \end{equation}
It follows that
\begin{equation}
\label{S7}
\begin{aligned}
G^{i\bi} \partial_i W \bpartial_i W
= 4 \fg_{1\bar{1}}^2 G^{i\bi} \nabla_i \fg_{1\bar{1}} \nabla_{\bi} \fg_{1\bar{1}},
\end{aligned}
\end{equation}
\begin{equation}
\label{S8}
\begin{aligned}
G^{i\bi} \partial_i (g_{1\bar{1}}^{-1} \fg_{1\bar{1}}) \bpartial_i (g_{1\bar{1}}^{-1} \fg_{1\bar{1}})
= G^{i\bi} \nabla_i \fg_{1\bar{1}} \nabla_{\bi} \fg_{1\bar{1}}
\end{aligned}
\end{equation}
and by Cauchy-Schwarz inequality,
\begin{equation}
\label{S9}
 \begin{aligned}
G^{i\bi} \bpartial_i \partial_i W
\geq \,& 2 \fg_{1\bar{1}} G^{i\bi} \nabla_{\bi} \nabla_i \fg_{1\bar{1}}
    + 2 G^{i\bi}  \nabla_i \fg_{1\bar{1}} \nabla_{\bi} \fg_{1\bar{1}} \\
  + \,& \sum_{l > 1} G^{i\bi} \nabla_i \fg_{1\bl} \nabla_{\bi} \fg_{l\bar{1}}
  + \frac{1}{2} \sum_{l > 1} G^{i\bi} \nabla_i \fg_{1\bl} \nabla_{\bi} \fg_{l\bar{1}}
   - C \fg_{1\bar{1}}^2 \sum G^{i\bi},
 \end{aligned}
 \end{equation}
\begin{equation}
\label{S10}
 \begin{aligned}
G^{i\bi} \bpartial_i  \partial_i (g_{1\bar{1}}^{-1} \fg_{1\bar{1}})
\geq \,& G^{i\bi} \nabla_{\bi} \nabla_i \fg_{1\bar{1}}
  - \frac{\gamma}{8 \fg_{1\bar{1}}} \sum_{l > 1} G^{i\bi}
            \nabla_i \fg_{1\bl} \nabla_{\bi} \fg_{l\bar{1}}
        - C \fg_{1\bar{1}} \sum G^{i\bi}.
 \end{aligned}
 \end{equation}

We derive from \eqref{S2}, \eqref{S3} and 
 \eqref{S4}-\eqref{S10} that
  \begin{equation}
\label{S11}
 \begin{aligned}
\nabla_i \fg_{1\bar{1}} + \fg_{1\bar{1}} \partial_i \eta = 0, \;
\nabla_{\bi} \fg_{1\bar{1}} + \fg_{1\bar{1}} \bpartial_i \eta = 0
\end{aligned}
\end{equation}
and
 \begin{equation}
\label{S12}
\begin{aligned}
0  \geq \,& \frac{1}{\fg_{1\bar{1}}}  G^{i\bi} \nabla_{\bi} \nabla_i \fg_{1\bar{1}}
 - \frac{1}{\fg_{1\bar{1}}^2} G^{i\bi}  \nabla_i \fg_{1\bar{1}} \nabla_{\bi} \fg_{1\bar{1}}
   + G^{i\bi} \bpartial_i \partial_i \eta \\
 & + \frac{\gamma}{\fg_{1\bar{1}}^2} \sum_{l > 1} G^{i\bi} \nabla_i \fg_{1\bl} \nabla_{\bi} \fg_{l\bar{1}} + \frac{\gamma}{16 \fg_{1\bar{1}}^2} \sum_{l > 1} G^{i\bi}
            \nabla_i \fg_{1\bl} \nabla_{\bi} \fg_{l\bar{1}}
        - C \sum G^{i\bi}.
\end{aligned}
\end{equation}

Now using \eqref{commute},

\begin{equation}
\label{S13}
 \begin{aligned}
 \nabla_{\bi} \nabla_i \fg_{1\bar{1}}  -  \nabla_{\bar{1}} \nabla_1 \fg_{i\bi}
   = \,&  R_{i\bi 1\bar{1}} \fg_{1\bar{1}} - R_{1\bar{1} i\bi} \fg_{i\bi}
         - T_{i1}^l \nabla_{\bi} \fg_{l\bar{1}}  - \ol{T_{i1}^l} \nabla_i \fg_{1\bl} \\
     &  - T_{i1}^l  \ol{T_{i1}^l} \fg_{l\bl} + H_{i\bi}
  \end{aligned}
 \end{equation}
  where
\[ \begin{aligned}
    H_{i\bi} = \,& \nabla_{\bi} \nabla_i X_{1\bar{1}}
                       -  \nabla_{\bar{1}} \nabla_1 X_{i\bi}
      - 2 \fRe\{T_{i1}^l \nabla_{\bi} X_{l\bar{1}}\} 
      + R_{i\bi 1\bl} X_{l\bar{1}} - R_{1\bar{1} i\bl} X_{l\bi}
       - T_{i1}^j  \ol{T_{i1}^l} X_{j\bl}.
  \end{aligned} \]
 It follows from Schwarz inequality that 
\begin{equation}
\label{S14}
 \begin{aligned}
G^{i\bi} \nabla_{\bi} \nabla_i \fg_{1\bar{1}}
   \geq \,& G^{i\bi} \nabla_{\bar1} \nabla_1\fg_{i\bi}
         -  \frac{\gamma}{32 \fg_{1\bar{1}}} 
             \sum\limits_{l>1}G^{i\bi}  \nabla_i \fg_{1\bl} \nabla_{\bi} \fg_{l\bar{1}} \\
       &  - C \fg_{1\bar{1}} \sum G^{i\bi} + G^{i\bi} H_{i\bi}.
 \end{aligned}
 \end{equation}
 
 \noindent We also have at $(z_0,t_0)$,
 \begin{equation}\label{S15}
     \begin{aligned}
     0\leq \partial_tQ&=(1+\gamma)\eta_t+(1+\gamma)\frac{\partial_t\fg_{1\bar{1}}}{\fg_{1\bar{1}}}\\
     &\leq(1+\gamma)\left(\eta_t+\frac{1}{\fg_{1\bar 1}}\left(G^{i\bi}\nabla_{\bar 1}\nabla_{1}\fg_{\bi i}-\nabla_1\nabla_{\bar{1}}\psi+\partial_t X_{1\bar{1}}\right)\right)
     \end{aligned}
 \end{equation}
 
 \noindent where we used \eqref{diff-pde}. Combining \cref{S12,S14,S15},
 
 \begin{equation}
 \label{S16}
     \begin{aligned}
     0\geq G^{i\bi}\bpartial_i\partial_i\eta-\eta_t+\frac{G^{i\bi}H_{i\bi}}{\fg_{1\bar{1}}}-C\sum G^{i\bi}-\frac{G^{i\bi}\nabla_i\fg_{1\bar{1}}\nabla_{\bi}\fg_{1\bar{1}}}{\fg_{1\bar{1}}^2}+\frac{\nabla_1\nabla_{\bar{1}}\psi}{\fg_{1\bar{1}}}-\frac{\partial_t X_{1\bar{1}}}{\fg_{1\bar{1}}}
     \end{aligned}
 \end{equation}
 
 \noindent It follows using \eqref{S11} that 
 
 \begin{equation}
 \label{S17}
     \begin{aligned}
     \fg_{1\bar{1}}\left(G^{i\bi}\bpartial_i\partial_i\eta-G^{i\bi}\partial_i\eta\bpartial_{i}\eta-\eta_t\right)\leq -G^{i\bi}H_{i\bi}+\partial_{t}X_{1\bar{1}}-\nabla_1\nabla_{\bar 1}\psi+C\fg_{1\bar{1}}\sum G^{i\bi}
     \end{aligned}
 \end{equation}

 \noindent By direct calculations (see e.g. \cite{GN21}, \cite{GQY19}) and using \cref{S11,diff-pde},
\begin{equation}
\label{S18}
\begin{aligned}
G^{i\bi} H_{i \bi} 
   \geq \,&  2 G^{i\bi} \fRe\{X_{1 \bar{1}, \zeta_{\alpha}}
                     \nabla_{\alpha} \nabla_{\bi} \nabla_i   \phi\}
                 -  2 G^{i\bi} \fRe\{X_{i \bi, \zeta_{\alpha}} \nabla_{\alpha} 
                     \nabla_{\bar{1}} \nabla_1 \phi\} \\
           & - C \fg_{1\bar{1}}^2 \sum G^{i\bi}  
               - C \sum_{i, k} G^{i\bi} |\nabla_i  \nabla_{k} \phi|^2
               - C \sum_k |\nabla_1 \nabla_k \phi|^2 \sum G^{i\bi}  \\
 \geq \,& 2 \fRe\{X_{1 \bar{1}, \zeta_{\alpha}}  (\nabla_{\alpha} \psi+\phi_{t\alpha})\}
          + 2 \fg_{1\bar{1}} G^{i\bi} \fRe\{X_{i \bi, \zeta_{\alpha}}  \nabla_{\alpha} \eta\} 
          - C |A|^2 \sum G^{i\bi}.
             \end{aligned}
 \end{equation}  
where we denote
\[ |A_i|^2 = \fg_{i\bi}^2 + \sum_{k} |\nabla_i \nabla_k \phi|^2, \;\; 
    |A|^2 =  \sum |A_i|^2. \]
Next, 
\begin{equation}
\label{S19}
\nabla_{\alpha} \psi = \psi_{\alpha} + \psi_{\phi} \nabla_{\alpha} \phi 
        + \psi_{\zeta_\beta} \partial_{\alpha} \partial_{\beta} \phi 
        + \psi_{\bzeta_\beta} \partial_{\alpha} \bpartial_{\beta} \phi 
        \end{equation}        

\begin{equation}
\label{S20}
\begin{aligned}
 \nabla_{\bar{1}} \nabla_1 \psi 
     \geq \,& \psi_{\zeta_\alpha} \nabla_{\alpha} \nabla_{\bar{1}} \nabla_1 \phi
                  + \psi_{\bzeta_\alpha} \nabla_{\balpha} \nabla_{\bar{1}} \nabla_1  \phi 
                  -  C |A_1|^2 \\ 
     \geq \,& - \fg_{1\bar{1}} \psi_{\zeta_\alpha} \nabla_{\alpha} \eta
                  - \fg_{1\bar{1}} \psi_{\bzeta_\alpha} \nabla_{\balpha} \eta -  C |A|^2. 
\end{aligned}
\end{equation}   

\noindent and

\begin{equation}\label{S21}
    \begin{aligned}
    \partial_tX_{1\bar{1}}=X_{1\bar{1},\phi}\phi_t+2\fRe\{X_{1\bar{1},\zeta_{\alpha}}\phi_{\alpha t}\}
    \end{aligned}
\end{equation}
 
 \noindent Plug in \cref{S18,S19,S20,S21} in \eqref{S17},

 \begin{equation}\label{S22}
 \begin{aligned}
     \fg_{1\bar{1}}\left(G^{i\bi}\bpartial_i\partial_i\eta-G^{i\bi}\partial_i\eta\bpartial_{i}\eta-\eta_t\right)\leq&-\nabla_{1}\nabla_{\bar{1}}\psi-2\fRe\{X_{1\bar{1},\zeta_{\alpha}}\nabla_{\alpha}\psi\}\\
     &-2\fg_{1\bar{1}}G^{i\bi}\fRe\{X_{i\bi,\zeta_{\alpha}}\nabla_{\alpha}\phi\}+C|A|^2\sum G^{i\bi}+X_{1\bar{1},\phi}\phi_t\\
     \leq&C\left(\fg_{1\bar{1}}|\nabla\phi|^2+|A|^2\right)\left(1+\sum G^{i\bi}\right)
 \end{aligned}
 \end{equation}
 
 Let $\eta=-\log{h}$, where $h=1-\gamma|\nabla \phi|^2$. We choose $\gamma$ small enough to satisfy $2\gamma|\nabla \phi|^2\leq 1$. 
 
 By straightforward calculations,
\begin{equation}
\label{S23}
\partial_i  |\nabla \phi|^2
    =  \nabla_k \phi \nabla_i \nabla_{\bk} \phi + \nabla_{\bk} \phi \nabla_i \nabla_k \phi
\end{equation}
and 
\begin{equation}
\label{S24}
  \begin{aligned}
 \bpartial_i \partial_i |\nabla \phi|^2 
   = \,& \nabla_i \nabla_{\bk} \phi \nabla_k \nabla_{\bi} \phi 
            +  \nabla_i \nabla_{k} \phi \nabla_{\bi} \nabla_{\bk} \phi \\
         &   + \nabla_{\bk} \phi \nabla_{\bi} \nabla_i \nabla_k \phi
            + \nabla_{k} \phi \nabla_{\bi} \nabla_i \nabla_{\bk} \phi \\
   = \,& \nabla_i \nabla_{\bk} \phi \nabla_k \nabla_{\bi} \phi 
            +  \nabla_i \nabla_{k} \phi \nabla_{\bi} \nabla_{\bk} \phi \\
         &   + \nabla_{\bk} \phi \nabla_k \nabla_{\bi} \nabla_i \phi 
            +  \nabla_k \phi \nabla_{\bk} \nabla_{\bi} \nabla_i \phi\\
         &   + R_{i\bi k\bl} \nabla_l \phi \nabla_{\bk} \phi
            - T^{l}_{ik} \nabla_{l\bi} \phi \nabla_{\bk} \phi 
            - \ol{T^{l}_{ik}} \nabla_{i\bl} \phi \nabla_k  \phi \\
     \geq  \,& (1 - \gamma) |A_i|^2 + \nabla_{\bk} \phi \nabla_k \nabla_{\bi} \nabla_i \phi 
                   + \nabla_k \phi \nabla_{\bk} \nabla_{\bi} \nabla_i \phi    - C |\nabla \phi|^2
\end{aligned}
\end{equation}


\begin{equation}\label{S25}
\begin{aligned}
G^{i\bi}(\bpartial_i\partial_i\eta-\bpartial_i\eta\partial_i\eta)=&\frac{\gamma}{h}G^{i\bi}\bpartial_i\partial_i|\nabla \phi|^2\\
\geq& G^{i\bi}\frac{\gamma(1-\gamma)}{h}|A_i|^2+\frac{\gamma}{h}G^{i\bi}\nabla_{\bk}\phi\nabla_k\nabla_\bi\nabla_i\phi +\frac{\gamma}{h}G^{i\bi}\nabla_k\phi\nabla_\bk\nabla_\bi\nabla_i\phi\\
&-C|\nabla \phi|^2\sum G^{i\bi}\\
\geq&\gamma(1-\gamma)\sum G^{i\bi}|A_i|^2+\frac{\gamma}{h}G^{i\bi}\nabla_{\bk}\phi\nabla_k\nabla_\bi\nabla_i\phi\\ &+\frac{\gamma}{h}G^{i\bi}\nabla_k\phi\nabla_\bk\nabla_\bi\nabla_i\phi-C|\nabla \phi|^2\sum G^{i\bi}
\end{aligned}
\end{equation}

We also have,

\begin{equation}\label{S26}
\begin{aligned}
\eta_t=\gamma \frac{\partial_t|\nabla \phi|^2}{h}
\end{aligned}
\end{equation}

\noindent and,

\begin{equation}\label{S27}
    \begin{aligned}
    \partial_t|\nabla u|^2=&2\fRe\{\partial_k\phi\bpartial_k\phi_t\}\\
    =&2\fRe\{\phi_k(G^{i\bi}\nabla_\bk\nabla_i\nabla_\bi\phi+G^{i\bi}\nabla_\bk X_{i\bi}-\partial_{\bk}\psi)\}
    \end{aligned}
\end{equation}

 Derive using \eqref{S25}, \eqref{S26} and \eqref{S27},
\begin{equation}\label{S28}
    \begin{aligned}
    G^{i\bi}(\bpartial_i\partial_i\eta-\bpartial_i\eta\partial_i\eta)-\eta_t\geq&\gamma(1-\gamma)\sum G^{i\bi}|A_i|^2 -C|\nabla \phi|^2\sum G^{i\bi}\\
    &-\frac{2\gamma}{h}\fRe\{\phi_k(G^{i\bi}\nabla_\bk X_{i\bi}-\partial_\bk\psi)\}\\
    \geq& (\gamma(1-\gamma)-4C\gamma^2)\sum G^{i\bi}|A_i|^2 -C|\nabla \phi|^2\sum G^{i\bi}
    \end{aligned}
\end{equation}

Further requiring that $\gamma$ is small enough to satisfy $\gamma\leq\dfrac{1}{2+4C}$, using \eqref{S22}
 \begin{equation}
\label{S29}
\begin{aligned}
\gamma^2  \fg_{1\bar{1}}
    \sum G^{i\bi}|A_i|^2 
    \leq   C |A|^2 \Big(1 + \sum G^{i\bi}\Big).
\end{aligned}
\end{equation}

From Lemma~\ref{lemma 3} we have the following key inequality for
each $i \geq 1$,
\begin{equation}
 \label{S30}
 G^{i\bi} = \sum_{l=1}^N f_{\Lambda_l} \frac{\partial \Lambda_l}{\partial \lambda_i} 
               = \sum_{i \in I} f_{\Lambda_I} \geq c_0 \sum_{l=1}^N f_{\Lambda_l}, 
\end{equation}                
where the sum $\sum_{i \in I} f_{\Lambda_I}$ is taken over all $I \in \fI_K$ with $i \in I$. 
Note that when $\fg_{i\bj}$ is diagonal,  so is $G^{i\bj}$. 

Consequently, by \eqref{S29} we obtain
  \begin{equation}
 \label{S31}
\begin{aligned}
\frac{c_0 \gamma^2 }{2} \fg_{1\bar{1}} |A|^2 \sum  f_{\Lambda_l}            
      \leq   C |A|^2 
\end{aligned}
\end{equation}  
provided that $ \fg_{1\bar{1}}$ is large enough. 
  
By the concavity of $f$ and using \eqref{P6}, we derive
\[  \begin{aligned}
    \sqrt{\fg_{1\bar1}}  \sum  f_{\Lambda_l}  
        = \,& \sqrt{\fg_{1\bar1}} \sum  f_{\Lambda_l}  
        - \sum  f_{\Lambda_l} \Lambda_l (\fg) + \sum  f_{\Lambda_l} \Lambda_l (\fg) \\
   \geq \,&  f (\sqrt{\fg_{1\bar1}} {\bf{1}}) - f (\Lambda (\fg)) 
        -C_0\sum f_{\Lambda_l} \\
        \geq \,&  \frac{c_0}{2} -C_0\sum f_{\Lambda_l}
  \end{aligned} \]
by assumption~\eqref{S0.3}, provided that $\fg_{1\bar1}$ is sufficiently large.
So from \eqref{S31}  we obtain
 \begin{equation}
 \label{S32}
\begin{aligned}
\fg_{1\bar{1}} |A|^2 \sum f_{\Lambda_l} + \sqrt{\fg_{1\bar{1}}}|A|^2            
      \leq   C |A|^2. 
\end{aligned}
\end{equation}

This gives the upper bound $\fg_{1\bar{1}}\leq C$. To obtain a lower bound for the eigenvalues $\fg_{i\bi}$, note that $\tr(\fg_{i\bi}+X)\geq 0$. This follows because the domain of $f$ is a symmetric cone in $\mathbb{R}^N$ with vertex at $0$ that contains the positive cone $\Gamma_N\subset\Gamma$.

\section{Gradient Estimates}

\label{G}
\setcounter{equation}{0}
\medskip

In this section we assume that $X$ and $\psi$ satisfies conditions \eqref{G0.1}, \eqref{G0.3}, \eqref{G0.5} and \eqref{G0.6}. With these assumptions in place, the gradient estimates can now be derived.

\begin{theorem}
Let $\phi\in C^{3,1}_{x,t}(M\times[0,T))$ be a solution of the equation \eqref{I1} in $M\times [0,T)$. Then 
\begin{equation}\label{G0}
    |\nabla \phi|_g^2\;\leq C(1+\sup{\phi}-\phi)
\end{equation}
for a uniform constant $C$ that depends on $\sup{|\phi_t|}$.
\end{theorem}

\begin{proof}
By adding a constant if necessary, we can assume without loss of generality that 
$$\sup\limits_{\partial \Gamma}f\leq 0 < \psi$$
Let $P= \eta+\log{|\nabla \phi|^2}$ where $\eta$ is a function of $\phi$ to be chosen later. Assume that $P$ attains maximum at the point $(z_0,t_0)\in M\times [0,T)$ and $|\nabla \phi|\geq 1$ at this point. We also choose local coordinates around $z_0$ so that $g_{i\bar j}=\delta_{i\bar j}$, $T^k_{ij}=2\Gamma^k_{ij}$ and $\fg_{i\bj}$ are diagonal at $z_0$. We have at $(z_0,t_0)$,

\begin{equation}\label{G1}
    \begin{aligned}
        &\partial_i|\nabla \phi|^2+|\nabla \phi|^2\partial_i\eta=0\\
        &\bpartial_i|\nabla \phi|^2+|\nabla \phi|^2\bpartial_i\eta=0
    \end{aligned}
\end{equation}

\noindent and,

\begin{equation}\label{G2}
\begin{aligned}
    G^{i\bi}\bpartial_i\partial_iP-\partial_tP
    =&G^{i\bi}\frac{\bpartial_i\partial_{i}|\nabla \phi|^2}{|\nabla \phi|^2}-G^{i\bi}\frac{\bpartial_i|\nabla \phi|^2\partial_i|\nabla \phi|^2}{|\nabla \phi|^4}+
    G^{i\bi}\bpartial_i\partial_i\eta\\
    &-\frac{\partial_t|\nabla \phi|^2}{|\nabla \phi|^2}-\partial_t\eta\leq 0
\end{aligned}
\end{equation}

Define $|Q_i|^2=\nabla_i\nabla_\bk\phi\nabla_k\nabla_\bi\phi+\nabla_i\phi\nabla_k\phi\nabla_\bi\phi\nabla_\bk\phi=\sum\limits_k(|\nabla_i\nabla_\bk\phi|^2+|\nabla_i\nabla_k\phi|^2)$. By Schwarz inequality,

\begin{equation}\label{G3}
    \begin{aligned}
        \bpartial_{i}|\nabla \phi|^2\partial_i|\nabla \phi|^2\leq 2|\nabla \phi|^2|Q_i|^2
    \end{aligned}
\end{equation}

\noindent and,

\begin{equation}\label{G4}
\begin{aligned}
    \bpartial_i\partial_i|\nabla \phi|^2=&\nabla_i\nabla_\bk\phi\nabla_k\nabla_{\bar{i}}\phi+\nabla_i\nabla_k\phi\nabla_\bi\nabla_\bk\phi\\
    &+\nabla_\bk\phi\nabla_k\nabla_\bi\nabla_i\phi+\nabla_k\phi\nabla_\bk\nabla_\bi\nabla_i\phi\\
    &+R_{i\bi k \bl}\nabla_l\phi \nabla_\bk\phi-T^l_{ik}\nabla_{l\bi}\phi\nabla_\bk\phi-\overline{T^i_{ik}}\nabla_{i\bl}\phi\nabla_k\phi\\
    &\geq (1-\gamma)|Q_i|^2+\nabla_\bk\phi\nabla_k\nabla_\bi\nabla_i\phi+\nabla_k\phi\nabla_\bk\nabla_\bi\nabla_i\phi-C|\nabla \phi|^2
\end{aligned}
\end{equation}

\noindent where $0<\gamma<\dfrac{1}{6}$. Also, 

\begin{equation}\label{G5}
\begin{aligned}
   G^{i\bi}\nabla_k\nabla_\bi\nabla_i\phi=G^{i\bi}(\nabla_k\fg_{i\bi}-\nabla_k X_{i\bi})=\nabla_k\psi+\nabla_k\phi_t-G^{i\bi}\nabla_kX_{i\bi} 
\end{aligned}
\end{equation}
\noindent Hence,

\begin{equation}\label{G6}
    \begin{aligned}
        G^{i\bi}\bpartial_i\partial_i|\nabla\phi|^2\geq& G^{i\bi}(1-\gamma)|Q_i|^2-C|\nabla \phi|^2\sum G^{i\bi}+R\\
        &+2\fRe\{\nabla_k\phi_t\nabla_\bk\phi\}
    \end{aligned}
\end{equation}

\noindent where $R=2\fRe\{(\nabla_k\psi-G^{i\bi}\nabla_kX_{i\bi})\nabla_{\bk}\phi\}$.

\begin{equation}\label{G6.1}
    \frac{\partial_t|\nabla\phi|^2}{|\nabla\phi|^2}=\frac{2}{|\nabla\phi|^2}\fRe\{\nabla_k\phi_t\nabla_\bk\phi\}
\end{equation}

Combine equations \eqref{G1}, \eqref{G2}, \eqref{G6} and \eqref{G6.1} to get cancellation of the terms involving $|Q_i|^2$ and $\nabla_k\phi_t$.

\begin{equation}\label{G7}
    \begin{aligned}
        G^{i\bi}\bpartial_i\partial_i\eta-\frac{1+\gamma}{2}G^{i\bi}\bpartial_i\eta\partial_i\eta\leq& -\frac{R}{|\nabla \phi|^2}+C\sum G^{i\bi}+\eta_t
    \end{aligned}
\end{equation}

Now we choose $\eta=-\log{h}$, where $h=1+\sup\limits_{M\times [0,T)}\phi-\phi$. So,

\begin{equation}\label{G8}
\begin{aligned}
    G^{i\bi}\bpartial_i\partial_i\eta=\frac{1}{h}G^{i\bi}\bpartial_i\partial_i\phi+\frac{1}{h^2}G^{i\bi}\bpartial_i\phi\partial_i\phi
\end{aligned}
\end{equation}

\noindent and,

\begin{equation}\label{G9}
\begin{aligned}
    G^{i\bi}\bpartial_i\eta\partial_i\eta= \frac{1}{h^2}G^{i\bi}\bpartial_i\phi\partial_i\phi
\end{aligned}
\end{equation}

From \eqref{S30} it follows that 

\begin{equation}\label{G10}
\begin{aligned}
    \frac{1}{h^2}G^{i\bi}\bpartial_i\phi\partial_i\phi-\frac{1+\gamma}{2}G^{i\bi}\bpartial_i\eta\partial_i\eta&= \frac{1-\gamma}{2h^2}G^{i\bi}\bpartial_i\phi\partial_i\phi\\
    &\geq \frac{c_1|\nabla\phi|^2}{4h^2}\sum G^{i\bi}
    \end{aligned}
\end{equation}

By concavity of $f$ and assumption \eqref{P6},

\begin{equation}\label{G11}
    \begin{aligned}
        |\nabla\phi|^2\sum G^{i\bi}&\geq f(|\nabla\phi|^2\mathbf{1})-f(\Lambda)+G^{i\bi}\fg_{i\bi}\\
        &\geq f(|\nabla\phi|^2\mathbf{1})-\psi-\phi_t-C\sum G^{i\bi}
    \end{aligned}
\end{equation}

Similarly,

\begin{equation}\label{G12}
    G^{i\bi}\bpartial_i\partial_i\phi=G^{i\bi}\fg_{i\bi}-G^{i\bi}X_{i\bi}\geq -G^{i\bi}X_{i\bi}-C\sum G^{i\bi}
\end{equation}

Combining the above inequalities and using \eqref{G7} we derive,

\begin{equation}\label{G13}
    \begin{aligned}
        \frac{c_1|\nabla\phi|^2}{8h^2}\sum G^{i\bi}+\frac{c_1}{8h^2}f(|\nabla\phi|^2\mathbf 1)&\leq -\frac{1}{h}G^{i\bi}\bpartial_i\partial_i\phi+\frac{c_1(\psi+\phi_t)}{8h^2}+C\sum G^{i\bi}-\frac{R}{|\nabla\phi|^2}+\eta_t\\
        &\leq \frac{1}{h}G^{i\bi}X_{i\bi}
        +\frac{c_1(\psi+\phi_t)}{8h^2}-\frac{R}{|\nabla\phi|^2}+\frac{\phi_t}{h}+C\sum G^{i\bi}
    \end{aligned}
\end{equation}

Using \eqref{G1} and chain rule we obtain,

\begin{equation}\label{G14}
    \begin{aligned}
        \fRe\{\nabla_k\psi\nabla_\bk\phi\}&=\psi_{\phi}|\nabla \phi|^2+\fRe\{\psi_k\nabla_\bk\phi+\psi_{\zeta_{\alpha}}\partial_{\alpha}|\nabla\phi|^2+\psi_{\zeta_{\alpha}}\Gamma_{\alpha k}^l\nabla_l\phi\nabla_{\bk}\phi\}\\
        &=|\nabla\phi|^2(\psi_{\phi}-\fRe\{\psi_{\zeta_{\alpha}}\partial_{\alpha}\eta\})+\fRe\{\psi_k\nabla_{\bk}\phi+\psi_{\zeta_{\alpha}}\Gamma_{\alpha k}^l\nabla_l\phi\nabla_\bk\phi\}\\
        &=|\nabla\phi|^2A
    \end{aligned}
\end{equation}
\noindent where

$$ A=\psi_{\phi}-\frac{1}{h}\fRe\{\psi_{\zeta_{\alpha}}\partial_{\alpha}\phi\}+\frac{1}{|\nabla\phi|^2}\fRe\{\psi_k\nabla_{\bk}\phi+\psi_{\zeta_{\alpha}}\Gamma_{\alpha k}^l\nabla_l\phi\nabla_\bk\phi\} $$

\noindent Similarly,
\begin{equation}\label{G15}
    \begin{aligned}
        G^{i\bi}\fRe\{\nabla_\bk\phi\nabla_kX_{i\bi}\}=|\nabla \phi|^2B
    \end{aligned}
\end{equation}
\noindent where 

$$B=G^{i\bi}X_{i\bi,\phi}-\frac{1}{h}G^{i\bi}\fRe\{X_{i\bi,\zeta_{\alpha}}\partial_{\alpha}\phi\}+\frac{1}{|\nabla\phi|^2}G^{i\bi}\fRe\{(X_{i\bi, k}+X_{i\bi,\zeta_{\alpha}}\Gamma^l_{\alpha k}\nabla_l\phi)\nabla_{\bk}\phi\}$$

By assumptions \eqref{G0.1} and \eqref{G0.3},

\begin{equation}\label{G15.1}
    \begin{aligned}
        \frac{1}{h}G^{i\bi}X_{i\bi}+\frac{c_1\psi}{8h^2}-\frac{R}{|\nabla\phi|^2}\leq C H\sum G^{i\bi}+CE+ C\left(1+\sum G^{i\bi}\right)
    \end{aligned}
\end{equation}

\noindent where
$$E=|\nabla_z\psi||\nabla \phi|^{-1}+(\psi_{\phi})^-+\psi^{+}+|D_{\zeta}\psi||\nabla\phi|\leq \varrho_0f(|\nabla\phi|^2\mathbf{1})+\varrho_1(z,\phi)$$

\noindent by \eqref{G0.3}, \eqref{G0.4}, \eqref{G0.6} and,

$$H=|\nabla_zX||\nabla \phi|^{-1}+\tr{X^+}+\tr(D_{\phi}X)^++|D_{\zeta}X||\nabla \phi|\leq \varrho_0|\nabla \phi|^2+\varrho_1$$

\

\noindent by \eqref{G0.1}, \eqref{G0.2}, \eqref{G0.5}. Use these inequalities to estimate the LHS of \eqref{G15.1} and plug into \eqref{G13} to obtain the bound $|\nabla \phi|^2\leq C$. From $P(z,t)\leq P(z_0,t_0)\leq C$, the required estimate \eqref{G0} follows.

\end{proof}

As a consequence we can bound the oscillation of $\phi$.

\begin{corollary}\label{cor1}
For $\phi$ as above,
$$\left|\left(1+\sup{\phi}-\phi(x,t)\right)^{\frac{1}{2}}-\left(1+\sup{\phi}-\phi(y,s)\right)^{\frac{1}{2}}\right|\leq C d$$

\noindent for any $(x,t)$, $(y,s)$ in $M\times [0,T)$, where $d$ is the diameter of $M$. 
In particular,

\begin{equation}\label{G16}
    \sup{\phi}-\inf{\phi}\leq C\max\{d,d^2\}
\end{equation}

\end{corollary}

\begin{proof}
Follows directly from the gradient estimates by using mean value theorem.
\end{proof}

\

\section{Long-time existence of solutions}

\label{L}
\setcounter{equation}{0}
\medskip
We shall prove the first part of Theorem \ref{theorem-I1} now. Recall that the normalized solution $\bar{\phi}$ solves,
\begin{equation}
\label{L0.1}
\begin{aligned}
    &\frac{\partial \bar{\phi}}{\partial t}= f(\Lambda(\sqrt{-1}\partial\bpartial {\bar{\phi}} +X[{\phi}]))-\psi[{\phi}]-\int_M\frac{\partial\phi}{\partial t}\omega^n\\
    &\bar{\phi}(x,0)=\phi_0-\int_M\phi_0\omega^n
\end{aligned}
\end{equation}

\noindent where $\phi$ is a solution of \eqref{I1}.

 \
 
Since $\int_M\bar{\phi}\omega^n=0$, there must be a $y\in M$ such that $\bar{\phi}(y)=0$. Using \eqref{G16}, 

\begin{equation}\label{L1}
\begin{aligned}
  |\bar{\phi}(x)|=&|\bar{\phi}(x)-\bar{\phi}(y)|\\
    =&|\phi(x)-\phi(y)|\leq C\max\{d,d^2\}
\end{aligned}
\end{equation}

Thus we obtain a uniform estimate for the normalized solution $\bar{\phi}$.

\

The second order estimate derived in section \ref{S} implies that equation \eqref{I1} is uniformly parabolic. Hence by general parabolic theory, equation \eqref{I1} has an admissible solution for some time $[0,T)$, where $T>0$ is the maximum time for which solution exists. Combined with the uniform apriori estimate $|\bar{\phi}|_{2,\alpha}\leq C$ from the previous sections, it will follow that $T=\infty$. Note that here $C^{2,\alpha}$ estimate followed directly once the $C^2$ estimate is established as a consequence of the Evans-Krylov theorem for parabolic equations.

To show $T=\infty$, first extend the $C^{2,\alpha}$ estimate for $\bar{\phi}$ to a uniform $C^{\infty}$ estimate by the standard bootstrapping argument. We sketch the idea here. Differentiate \eqref{I1} with respect to $z_l$,

\begin{equation}\label{L2}
\frac{\partial\bar{\phi}_l}{\partial t}=G^{i\bj}\partial_i\bpartial_j\bar{\phi}_l+\chi_k[{\phi}](\bar{\phi}_l)_k+\chi_{\bar{k}}[{\phi}](\bar{\phi}_l)_{\bar k}+\chi_0[{\phi}]\bar{\phi}_l
\end{equation}

\noindent where the coefficient functions are as in section \ref{U}. This is a linear parabolic equation in $\bar{\phi}_l$ whose coefficients are in $C^{\alpha}$ with $\chi_0\leq 0$. Hence by parabolic Schauder estimates, we get that $|\bar{\phi}_l|_{2,\alpha}\leq C$ for a uniform constant $C$. Similarly $|\bar{\phi}_\bl|_{2,\alpha}\leq C$. By inductively applying this argument to higher derivatives it follows that $|\bar{\phi}|_{C^{\infty}(M)}\leq C$. Similarly we obtain $|\bar{\phi}_t|_{C^{\infty}(M)}\leq C$ by applying the same technique on \eqref{U4}.

\ 

To prove $T=\infty$, assume for contradiction that $T<\infty$. Then the solution $\bar{\phi}$ of \eqref{L0.1} can be extended to $T$ using the apriori estimates. Now \eqref{L0.1} with initial data $\bar{\phi}(.,T)$ is a parabolic PDE starting at time $T$ with smooth initial data. Hence the solution can be extended to $[0,T+\epsilon)$, for some $\epsilon>0$. This contradicts the maximality of $T$. Thus the solution exists for all time $[0,\infty)$. The long time existence of the solution $\phi$ also follows similarly after obtaining an estimate (possibly depending on $T$) for $\sup{|\phi|} $.

\end{proof}

\section{Harnack inequality}
\label{H}
\setcounter{equation}{0}
\medskip

In this section we will derive a Harnack inequality for the time derivative $\phi_t$ of solutions of \eqref{I1}. For this purpose, we extend the results of Gill \cite{Gill11} and Li-Yau \cite{LY86} to parabolic equations with lower order terms. More precisely, consider the following equation,

\begin{equation}
    \label{H1}
    \begin{aligned}
    \frac{\partial u}{\partial t}=  G^{i\bar j}\partial_i\partial_{\bar j}u+\chi_ku_k+\chi_{\bar{k}}u_{\bar k}+\chi_0 u
    \end{aligned}
\end{equation}

\noindent where $G^{i\bj}$, $\chi_k$, $\chi_{\bk}$ and $\chi_{0}$ are time-dependent functions with $G^{i\bar j}$ being $C_{x,t}^{3,1}$ and $\chi_k$, $\chi_{\bk}$, $\chi_{0}$  are assumed to be $C_{x,t}^{1,1}$.

Let $u$ be a positive solution of \eqref{H1} in $M\times [0, T)$ for some $T>0$. Define $f=\log u$ and $F=t(|\partial f|^2-\alpha f_t)$, where $|\partial f|^2=G^{i\bar j} f_i f_{\bar j}$ and $1<\alpha <2$.  $G^{i\bar j}$ is assumed to be uniformly elliptic with $0<\lambda|\xi|^2\leq G^{i\bar j}\xi_i\xi_{\bar j}\leq \Lambda |\xi|^2 $ in $M$ for any vector $\xi$. Also denote $\langle X,Y \rangle=G^{i\bar j} X_i Y_j$. All the norms and inner products in this section will be computed with respect to $G^{i\bj}$.\\

\begin{lemma}\label{lemma2.1}
Let $u\in C_{x,t}^{3,2}(M\times [0,T))$ be a positive solution of \eqref{H1} in $[0,T)$. Then for $t>0$
\begin{equation}
\label{H2}
\begin{aligned}
      |\partial f|^2-\alpha f_t\leq C_1+\frac{C_2}{t}  
\end{aligned}
\end{equation}

\noindent for some constants $C_1$ and $C_2$ that depends only on the coefficient functions $G^{i\bj}$, $\chi_k$, $\chi_{\bk}$ and $\chi_0$.
\end{lemma}

\begin{proof} In the following calculations $C$, $C_1$ and $C_2$ will denote generic constants that may change from line to line. 
We will apply maximum principle to $F$. Let $(x_0,t_0)$ be a point in $M\times (0,T']$ where F attains maximum. Here $0<T'<T$ is a fixed time. 

Then we have at $(x_0,t_0)$,

\begin{equation}
\label{H3}
\begin{aligned}
    \partial_k |\partial f|^2&=\alpha f_{kt}\\
    \partial_{\bar k} |\partial f|^2&=\alpha f_{\bar kt}
\end{aligned}
\end{equation}

From \eqref{H1} we derive

\begin{equation}
\label{H4}
\begin{aligned}
    G^{i\bar j}f_{i\bar j}-f_t=-|\partial f|^2-\chi_kf_k-\chi_{\bar{k}}f_{\bar k}-\chi_0
\end{aligned}
\end{equation}

Plugging this in $F$ gives,

\begin{equation}
\label{H5}
\begin{aligned}
    F&=-tG^{i\bar j}f_{i\bar j}+t(1-\alpha)f_t-t\chi_kf_k-t \chi_{\bar{k}}f_{\bar k}-t\chi_0
\end{aligned}
\end{equation}

Next compute $F_t$ and $G^{i\bar j}F_{i\bar j}$.
\begin{equation}
\label{H6}
\begin{aligned}
    F_t=|\partial f|^2-\alpha f_t +2t \fRe\langle\partial f,\partial f_t \rangle+t\partial_tG^{i\bar j}f_if_{\bar j}-\alpha tf_{tt} 
\end{aligned}
\end{equation}

\noindent and,

\begin{equation}
\label{H7}
\begin{aligned}
    G^{i\bar j}F_{i\bar j}&= tG^{i\bar j}\biggl[\partial_i\partial _{\bar j}G^{k\bar l}f_{k}f_{\bar l}+\partial_iG^{k\bar l}f_{k\bar j}f_{\bar l}+\partial_{i}G^{k\bar l}f_{k}f_{\bar l\bar j}+\partial_{\bar j}G^{k\bar l}f_{\bar l}f_{ki}+\partial_{\bar j}G^{k\bar l}f_{k}f_{\bar l i}\\ 
    &+G^{k\bar l}f_{ki}f_{\bar l\bar j}+G^{k\bar l}f_{k\bar j}f_{\bar l i}+G^{k\bar l}f_{ki\bar j}f_{\bar l}+G^{k\bar l}f_{k}f_{\bar l i \bar j}-\alpha f_{ti\bar j}\biggr]
\end{aligned}
\end{equation}

We now estimate all the terms in the above equation. Consider the first five terms in \eqref{H7}. By Cauchy-Schwarz inequality,

\begin{equation}
\label{H8}
\begin{aligned}
   tG^{i\bar j}\biggl[\partial_i\partial _{\bar j}G^{k\bar l}f_{k}f_{\bar l}&+\partial_iG^{k\bar l}f_{k\bar j}f_{\bar l}+\partial_{i}G^{k\bar l}f_{k}f_{\bar l\bar j}+\partial_{\bar j}G^{k\bar l}f_{\bar l}f_{ki}+\partial_{\bar j}G^{k\bar l}f_{k}f_{\bar l i}\biggr] \\
    &\leq C\left[ t|\partial f|^2+\frac{2t}{\epsilon}|\partial f|^2+ t\epsilon |\partial \bpartial f|^2+t\epsilon|\partial\partial f|^2\right] 
\end{aligned}
\end{equation}

\noindent where $\epsilon>0$ is a small constant to be chosen later. Here $|\partial \bpartial f|^2=G^{i\bar j}G^{k\bar l}f_{i\bar l}f_{k\bar j}$ and $|\partial\partial f|^2=G^{i\bar j}G^{k \bar l}f_{ik}f_{\bar j\bar l}$.

Write the third order terms in \eqref{H7} using \eqref{H5} as follows.

\begin{equation}
\label{H9}
\begin{aligned}
   tG^{i\bar j}G^{k\bar l}f_{ki\bar j}f_{\bar l}+tG^{i\bar j}G^{k\bar l}f_{k}f_{\bar l i \bar j} &= 2t \fRe \langle \partial f, \partial (G^{i\bar j}f_{i\bar j})\rangle-tG^{k\bar l}\partial_{k}G^{i\bar j}f_{i\bar j}f_{\bar l}-tG^{k\bar l}\partial_{\bar l}G^{i\bar j}f_{k}f_{i\bar j}\\
   &\geq 2t \fRe\langle\partial f,\partial (G^{i\bar j}f_{i\bar j}) \rangle-\frac{Ct}{\epsilon}|\partial f|^2-t\epsilon|\partial \bpartial f|^2\\
   &= -2\fRe\langle\partial f, \partial F \rangle+2t(1-\alpha)\fRe\langle \partial f, \partial f_t\rangle\\
   & -2t \fRe\langle\partial f, \partial [\chi_kf_k+\chi_{\bar k}f_{\bar k}+\chi_0] \rangle-\frac{Ct}{\epsilon}|\partial f|^2-t\epsilon|\partial \bpartial f|^2  
\end{aligned}
\end{equation}

We can write,

\begin{equation}
\label{H10}
\begin{aligned}
    \biggl|\fRe\langle\partial f, \partial [\chi_kf_k+\chi_{\bar k}f_{\bar k}+\chi_0] \rangle\biggr|&=\biggl| \fRe\langle \partial f,\partial \chi_kf_k\rangle+\fRe\langle \partial f,\chi_k\partial f_k\rangle+\fRe\langle \partial f,\partial \chi_{\bar k}f_{\bar k}\rangle \\
    &+\fRe\langle \partial f,\chi_{\bar k}\partial f_{\bar k}\rangle+\fRe\langle \partial f,\partial \chi_0\rangle\biggr|\\
    &\leq C\left(|\partial f|^2+|\langle \partial f,\partial\partial f\rangle|+|\langle \partial f,\partial \bpartial f\rangle|\right)\\
    &\leq \left(C+\frac{2C}{\epsilon}\right)|\partial f|^2+\epsilon |\partial \partial f|^2+\epsilon|\partial \bpartial f|^2 
\end{aligned}
\end{equation}

Now combining \eqref{H10}, \eqref{H9} and \eqref{H6},

\begin{equation}
\label{H11}
\begin{aligned}
    tG^{i\bar j}G^{k\bar l}f_{ki\bar j}f_{\bar l}+tG^{i\bar j}G^{k\bar l}f_{k}f_{\bar l i \bar j} &\geq -2\fRe\langle \partial f, \partial F\rangle -(\alpha -1)F_t+(\alpha -1)(|\partial f|^2-\alpha f_t)\\
    &-C_2t|\partial f|^2 -t\alpha(\alpha -1)f_{tt}-2t\left(C+\frac{3C}{\epsilon}\right)|\partial f|^2\\
    &-2t\epsilon |\partial\partial f|^2-3t\epsilon|\partial \bpartial f|^2
\end{aligned}
\end{equation}

To estimate the last term in \eqref{H7}, we differentiate \eqref{H5} wrt $t$.

\begin{equation}
\label{H12}
\begin{aligned}
    t\frac{\partial}{\partial t}(G^{i\bar j}f_{i\bar j})=\frac{F}{t}-F_t+t(1-\alpha)f_{tt}-t(\partial_t\chi_kf_k+\chi_kf_{kt}+\partial_t\chi_{\bar k}f_{\bar k}+\chi_{\bar k}f_{\bar{k}t}+\partial_t\chi_0)
\end{aligned}
\end{equation}

Use \eqref{H3} to control $f_{kt}$ and $f_{\bar k t}$ terms above.

\begin{equation}
\label{H13}
\begin{aligned}
    |\chi_kf_{kt}+\chi_{\bar k}f_{\bar k t}|=\frac{1}{\alpha}\left|\chi_k \partial_k|f|^2+\chi_{\bar k} \partial_{\bar k}|f|^2\right|\leq \frac{C}{\alpha\epsilon}|\partial f|^2+\frac{\epsilon}{2\alpha} |\partial \partial f|^2+\frac{\epsilon}{2\alpha}|\partial \bpartial f|^2
\end{aligned}
\end{equation}

Now estimate the last term in \eqref{H7} as follows.
\begin{equation}
\label{H14}
\begin{aligned}
    -\alpha tG^{i\bar j }f_{ti\bar j}=&\alpha t \partial_tG^{i\bar j}f_{i\bar j}-\alpha t\frac{\partial }{\partial t}(G^{i\bar j}f_{i\bar j})\\ 
    \geq& -\frac{Ct}{\epsilon}-t\epsilon |\partial \bpartial f|^2-\frac{\alpha}{t}F+\alpha F_t+t\alpha(\alpha -1)f_{tt}-t(C_1|\partial f|^2\\
    &+\epsilon |\partial \partial f|^2+\epsilon |\partial \bpartial f|^2+ C_2)
\end{aligned}
\end{equation}

\noindent where we used \eqref{H12} and \eqref{H13} in the last inequality. Combine \cref{H7,H8,H11,H14} to get

\begin{equation}
\label{H15}
\begin{aligned}
    G^{i\bar j}F_{i\bar j}\geq &F_t-2\fRe\langle \partial f,\partial F\rangle -(|\partial f|^2-\alpha f_t)-Ct|\partial f|^2+t(1-(5+C)\epsilon)|\partial \bpartial f|^2\\
    &+t(1-(3+C)\epsilon)|\partial \partial f|^2-Ct
\end{aligned}
\end{equation}

Choose $\epsilon=\dfrac{1}{2(5+C)}$. Also by \eqref{H4},

 \begin{equation}
 \label{H16}
 \begin{aligned}
     |\partial \bpartial f|^2&\geq \frac{1}{n}(G^{i\bar j}f_{i\bar j})^2=\frac{1}{n}(|\partial f|^2-f_t+(\chi_kf_k+\chi_{\bar k}f_{\bar k}+\chi_0))^2\\
     &\geq \frac{1}{2n}(|\partial f|^2-f_t)^2-C_1|\partial f|^2-C_2
 \end{aligned}
 \end{equation}
 
Plugging this above and using $\partial F=0$ and $G^{i\bar j}F^{i\bar j}-F_t\leq 0$ at $(x_0,t_0)$, we get

\begin{equation}
\label{H17}
\begin{aligned}
    0\geq-(|\partial f|^2-\alpha f_t)-Ct_0|\partial f|^2+\frac{t_0}{4n}(|\partial f|^2-f_t)^2-Ct_0
\end{aligned}
\end{equation}

The rest of the proof can be completed by splitting into two cases when $f_t(x_0, t_0)$ is non-negative and when it is negative, similar to \cite{Gill11}. For convenience, we provide the details here.

First assume that $f_t(x_0,t_0)\geq 0$, then we can deduce from the above equation,

\begin{equation}
    \label{H18}
    \begin{aligned}
        \frac{1}{4n}(|\partial f|^2-f_t)\left(|\partial f|^2-f_t-\frac{4n}{t_0}\right)\leq C_1|\partial f|^2+C_2
    \end{aligned}
\end{equation}

So it follows that,

\begin{equation}
\label{H19}
    |\partial f|^2-f_t\leq C_1|\partial f|+\frac{C_2}{t_0}+C_3
\end{equation}
Using Schwarz inequality we have,

\begin{equation}
\label{H20}
    C_1|\partial f|\leq \left(1-\frac{1}{\alpha}\right)|\partial f|^2+C_4
\end{equation}
Plug this in \eqref{H19} to get,

\begin{equation}
\label{H21}
    \frac{1}{\alpha}|\partial f|^2-f_t\leq C_1+\frac{C_2}{t_0}
\end{equation}
For any $x\in M$,

\begin{equation}
    \label{H22}
    \begin{aligned}
        F(x,T')&\leq F(x_0,t_0)\\
        &\leq C_1t_0+C_2\leq C_1T' +C_5
    \end{aligned}
\end{equation}

Now the result follows from the definition of $F$ and taking $T'=t$. For the case when $f_t(x_0,t_0)<0$, from \eqref{H17},

\begin{equation}
    \label{H23}
    \frac{t_0}{4n}|\partial f|^4-|\partial f|^2\leq C_1 t_0|\partial f|^2+C_2t_0-\alpha f_t
\end{equation}

Factor this to get,

\begin{equation}
    \label{H24}
    |\partial f|^2\left(\frac{1}{4n}|\partial f|^2-\frac{1}{t_0}-C_1\right)\leq C_2-\frac{\alpha}{t_0}f_t
\end{equation}

It follows that,

\begin{equation}
    \label{H25}
    |\partial f|^2\leq C_1+\frac{1}{t_0}+C\sqrt{-\frac{1}{t_0}f_t}\leq  C_1+\frac{C_2}{t_0}-\frac{1}{2}f_t
\end{equation}

By \eqref{H17} and using $f_t(x_0,t_0)<0$ we get

\begin{equation}
    \label{H26}
    \frac{1}{4n}(-f_t)\left(-f_t-\frac{4n\alpha}{t_0}\right)\leq C_1 |\partial f|^2+\frac{1}{t_0}|\partial f|^2+C_2
\end{equation}

This implies

\begin{equation}
    \label{H27}
    -f_t\leq \frac{4n\alpha}{t_0}+C_1|\partial f|+C\frac{|\partial f|}{\sqrt{t_0}}+C_2
\end{equation}

Applying Cauchy-Schwarz inequality to the above gives,

\begin{equation}
\label{H28}
    -f_t\leq C_1+\frac{C_2}{t_0}+\frac{|\partial f|^2}{2}
\end{equation}

Plug \eqref{H28} into \eqref{H25} to get,
\begin{equation}
    \label{H29}
    |\partial f|^2\leq C_1+\frac{C_2}{t_0}
\end{equation}

Using this in \eqref{H28} we deduce,

\begin{equation}
    \label{30}
    -\alpha f_t\leq C_1+\frac{C_2}{t_0}
\end{equation}

Adding the above two equations gives an estimate similar to \eqref{H21}.
\begin{equation}
    \label{31}
     |\partial f|^2-\alpha f_t\leq C_1+\frac{C_2}{t_0}
\end{equation}

Now the proof is completed in the same way as in the first case.

\end{proof}

We use this lemma to derive a Harnack inequality along the lines of Li and Yau.
\begin{theorem}\label{theorem-H}
Let $u$ be a solution of \eqref{H1} as in Lemma \ref{lemma2.1}. Then for $0<t_1<t_2$,

\begin{equation}
    \sup_{x\in M} u(x,t_1)\leq C(t_1,t_2)\inf_{x\in M}u(x,t_2)
\end{equation}
for
\begin{equation}
    C(t_1,t_2)=\left(\frac{t_2}{t_1}\right)^{C_2}\exp\left(\frac{C_3}{t_2-t_1}+C_1(t_2-t_1)\right)
\end{equation}

\noindent where $C_1$, $C_2$ and $C_3$ are constants depending only on the $C^{3,1}_{x,t}(M\times[0,T))$ norm of $G^{i\bar{j} }$ and $C^{1,1}_{x,t}(M\times[0,T))$ norms of $\chi_0$, $\chi_k$, $\chi_{\bar k}$.
\end{theorem}

\begin{proof}

Let $\gamma: [0,1]\to M$ be a unit speed curve such that $\gamma(0)=y$ and $\gamma(1)=x$. Then we define a path $\eta: [0,1]\to M\times [t_1,t_2]$ joining $(y,t_2)$ to $(x,t_1)$ by $\eta(s)=(\gamma(s), (1-s)t_2+st_1 ) $. We can write,

\begin{equation}
    \begin{aligned}
        \log\frac{u(x,t_1)}{u(y,t_2)}&=\int_0^1\frac{d}{ds}f(\eta(s))ds\\
        &=\int_0^1\langle \dot{\gamma},\partial f\rangle-(t_2-t_1)f_t\;ds\\
        &\leq\int_0^1 -\frac{t_2-t_1}{\alpha}\left(\frac{\alpha |\dot{\gamma}|}{t_2-t_1}-|\partial f|\right)^2+\frac{\alpha |\dot{\gamma}|^2}{t_2-t_1}+\frac{t_2-t_1}{\alpha}(|\partial f|^2-\alpha f_t) ds\\
        &\leq \int_0^1 \frac{C}{t_2-t_1}+C(t_2-t_1)\left(1+\frac{1}{(1-s)t_2+st_1}\right) ds\\
        &= C_1(t_2-t_1)+C_2\log\left(\frac{t_2}{t_1}\right)+\frac{C_3}{t_2-t_1}
    \end{aligned}
\end{equation}
\noindent where Lemma \ref{lemma2.1} is used in the fourth line. The final equation is obtained by taking exponentials on both sides followed by infimum in $y$ and supremum in $x$ over $M$.

\end{proof}

\section{Convergence of the solution}

\label{C}
\setcounter{equation}{0}
\medskip
In this section we assume that $X$ and $\psi$ are independent of $\phi$ (but still depends on $\partial \phi,\bpartial \phi$). 
To show convergence of the solution we will use a standard iteration argument for the oscillation of the solution. Define $u=\phi_t$ as before and consider the following functions. 

\begin{equation}\label{C1}
\begin{aligned}
    v_n(x,t)&=\sup\limits_{y\in M}u(y,n-1)-u(x,n-1+t)\\
    w_n(x,t)&=u(x,n-1+t)-\inf\limits_{y\in M}u(y,n-1)
    \end{aligned}
\end{equation}

\noindent The oscillation of $u$ is defined as a function of $t$ by $\omega(t):=\sup\limits_{x\in M}u(x,t)-\inf\limits_{x\in M}u(x,t)$. Then both $v_n$ and $w_n$ satisfy the following PDE.

\begin{equation}\label{C2}
\begin{aligned}
    \frac{\partial \varphi}{\partial t}(x,t)&=G^{i\bj}(x,n-1+t)\partial_i\partial_{\bj}\varphi+\chi_k(x,n-1+t)\partial_k\varphi+\chi_\bk(x,n-1+t)\partial_\bk \varphi 
\end{aligned}
\end{equation}

Note that $\chi_0=G^{i\bj}X_{i\bar j,\phi}-\psi_{\phi}\equiv 0$ by assumption. If $u(x,n-1)$ is not constant then $v_n$ is positive for some $x$ in $M$ at time $t=0$. This implies that $v_n$ is positive for all $t>0$ by the maximum principle. Likewise for $w_n$. So by applying Theorem \ref{theorem-H} to $v_n$ and $w_n$ with $t_1=\frac{1}{2}$ and $t_2=1$,

\begin{equation}\label{C3}
    \begin{aligned}
\sup\limits_{x\in M}u(x,n-1)-\inf\limits_{x\in M}u\left(x,n-\frac{1}{2}\right)&\leq C\left(\sup\limits_{x\in M}u(x,n-1)-\sup\limits_{x\in M}u(x,n)\right)\\
\sup\limits_{x\in M}u\left(x,n-\frac{1}{2}\right)-\inf\limits_{x\in M}u(x,n-1)&\leq C\left(\inf\limits_{x\in M}u(x,n)-\inf\limits_{x\in M}u(x,n-1)\right)
    \end{aligned}
\end{equation}
\noindent where $C:=C(\frac{1}{2},1)$. By adding the two equations above, we see that $\omega(t)$ satisfies the following recursion.

\begin{equation}\label{C4}
    \begin{aligned}
    \omega(n-1)+\omega\left(n-\frac{1}{2}\right)\leq C(\omega(n-1)-\omega(n))
    \end{aligned}
\end{equation}
It follows that $\omega(n)\leq \delta \omega(n-1)$ for some $\delta<1$ and by iterating we get that $\omega(t)\leq Ce^{-\beta }t$ for $\beta=-\log{\delta}$. If $u(x,n-1)$ is constant the same estimate holds by maximum principle applied to $v_n$. Fix $(x,t)\in M\times [0,\infty)$. Since $\int_M\frac{\partial \bar{\phi}}{\partial t}\omega^n=0$, there is a point $y\in M$ such that $\frac{\partial \bar{\phi}}{\partial t}(y,t)=0$. Hence,

\begin{equation}\label{C5}
    \begin{aligned}
    \left|\frac{\partial \bar{\phi}}{\partial t}(x,t)\right|=\left|\frac{\partial \phi}{\partial t}(x,t)-\frac{\partial \phi}{\partial t}(y,t)\right|\leq Ce^{-\beta}t
    \end{aligned}
\end{equation}

Now $h(t)=\bar{\phi}+\dfrac{Ce^{-\beta t}}{\beta}$ satisfies $\dfrac{\partial h}{\partial t}\leq 0$. So $h(t)$ is bounded and monotonically decreasing for each $x$. Denote the limit function by $\bar{\phi}_{\infty}$. From the definition of $h(t)$ it is clear that $\bar{\phi}$ converges pointwise in $x$ to the same function $\bar{\phi}_{\infty}$ as $t \to \infty$.

\

To show that the convergence is smooth, we assume for contradiction that there exists a sequence of times $\{t_l\}$ such that,

\begin{equation}\label{C6}
    |\bar{\phi}(.,t_l)-\bar{\phi}_{\infty}|_{C^k(M)}>\epsilon\; \forall \; l
\end{equation}

\noindent for some $k$.

Using the uniform estimates on the $C^{\infty}$-norm of $\bar{\phi}$, we can extract a subsequence $\{t_{l_m}\}$ along which $\bar{\phi}$ converges in $C^{\infty}$ to some smooth function $\hat{\phi}_{\infty}$. But then by pointwise convergence we have that $\hat{\phi}_{\infty}\equiv\bar{\phi}_{\infty}$, and hence \eqref{C6} is not possible.

Finally we prove the convergence in Theorem \ref{theorem-I1}. Take limit $t \to \infty$ in \eqref{L0.1}. By \eqref{C5} and the previous paragraph, it follows that

\begin{equation}
\label{C7}
\begin{aligned}
    f(\Lambda(\sqrt{-1}\partial\bpartial \bar{\phi}_{\infty} +X[\bar{\phi}_{\infty}]))=\psi[\bar{\phi}_{\infty}]+a
\end{aligned}
\end{equation}

\noindent where, 
$$a=\lim\limits_{t\to \infty}\int_M\frac{\partial{\phi}}{\partial t}\omega^n$$

\clearpage

\bibliographystyle{plain}

\begin{thebibliography}{}
\bibitem{CNS85}
L. Caffarelli, L. Nirenberg and J. Spruck,
{\em The Dirichlet problem for nonlinear second order elliptic equations, III: Functions of the eigenvalues of the Hessian}, Acta Mathematica {\bf 155} (1985),  261 -- 301.



\bibitem{Cao1985}
H.D. Cao,
{\em Deformation of Kähler metrics to Kähler-Einstein metrics on compact Kähler manifolds},
Inventiones mathematicae {\bf 81} (1985), 359 -- 372.


\bibitem{GGQ21}
Mathew George, Bo Guan and Chunhui Qiu,
{\em Fully nonlinear elliptic equations on Hermitian manifolds for symmetric functions of partial Laplacians},
arXiv:2110.00490.


\bibitem{Gill11}

Matt Gill,
{\em Convergence of the parabolic complex Monge-Amp\`ere equation on compact Hermitian manifolds},
Communications in Analysis and Geometry {\bf 19(2)} (2011), 277 -- 303,
arXiv:1009.5756.

\bibitem{G14}
Bo Guan,
{\em Second-order estimates and regularity for fully nonlinear elliptic equations on Riemannian manifolds},
Duke Mathematical Journal {\bf 163(8)} (2014), 1491 -- 1524,
arXiv:1211.0181.


\bibitem{GN21}
Bo Guan and Xiaolan Nie,
{\em Second order estimates for fully nonlinear elliptic equations with gradient terms on Hermitian manifolds},
arXiv:2108.03308.

\bibitem{GQY19}
Bo Guan, Chunhui Qiu and Rirong Yuan,
{\em Fully nonlinear elliptic equations for conformal deformations of Chern–Ricci forms},
Advances in Mathematics {\bf 343} (2019), 538 -- 566,
arXiv:2108.03308.

\bibitem{LY86}
Peter Li and Shing Tung Yau,
{\em On the parabolic kernel of the Schrödinger operator},
Acta Mathematica {\bf 156} (1986), 153 -- 201.

\bibitem{Lieberman}
G.M. Lieberman,
{\em Second Order Parabolic Differential Equations},
World Scientific (1996).

\bibitem{LT94}
Mi Lin and N. S. Trudinger,
{\em On some inequalities for elementary symmetric functions},
Bulletin of the Australian Mathematical Society {\bf 50} (1994), 317 -- 326.

\bibitem{ST09}
Jeffrey Streets and Gang Tian,
{\em Hermitian Curvature Flow},
Journal of the European Mathematical Society {\bf 13(3)} (2011), 601 -- 634,
arXiv:0804.4109.

\bibitem{STW17}
Gábor Székelyhidi, Valentino Tosatti and Ben Weinkove, 
{\em Gauduchon metrics with prescribed volume form},
Acta Mathematica {\bf 219} (2017), 181 -- 211,
arXiv:1503.04491.

\bibitem{TW17}
Valentino Tosatti and Ben Weinkove,
{\em The Monge-Ampère equation for (n-1)-plurisubharmonic functions on a compact Kähler manifold},
Journal of the American Mathematical Society {\bf 30} (2016), 311 -- 346,
arXiv:1305.7511.













\end{thebibliography}

\end{document}